\documentclass[12pt]{article}
\usepackage{amsfonts,latexsym,rawfonts,amsmath,amssymb,amsthm}
\usepackage{amsmath,amssymb,amsfonts,latexsym,lscape,rawfonts}
\usepackage[cm]{fullpage}
\usepackage{amsmath,amscd, float,times,rotating}
\usepackage{pb-diagram}
\usepackage{color}
\usepackage{hyperref}
\numberwithin{equation}{section}

 \newtheorem{thm}{Theorem}[section]
 
 \newtheorem{lem}[thm]{Lemma}
 
 \theoremstyle{definition}
 
 \theoremstyle{remark}
 \newtheorem{rem}[thm]{Remark}
 
 \numberwithin{equation}{section}

\begin{document}

\title
 {Gap theorems for K\"ahler-Ricci solitons}

\author{Haozhao Li}

\date{}

\maketitle

\newcommand{\oo}{\omega}
\def\pbp{\sqrt{-1}\partial\bar\partial}
\newcommand{\Na}{\nabla}
\newcommand{\ee}{\epsilon}
\newcommand{\la}{\lambda}
\newcommand{\La}{\Lambda}
\newcommand{\beqs}{\begin{eqnarray*}}
\newcommand{\eeqs}{\end{eqnarray*}}
\def\ri{\rightarrow}
\newcommand{\beqn}{\begin{eqnarray}}
\newcommand{\eeqn}{\end{eqnarray}}

\noindent {\bf Abstract} \ In this paper, we prove that a gradient
shrinking compact K\"ahler-Ricci soliton cannot have too large Ricci
curvature unless it is K\"ahler-Einstein.

\section{Introduction}
The main purpose of this paper is to give two gap theorems on
gradient shrinking K\"ahler-Ricci solitons with positive Ricci
curvature on a compact K\"ahler manifold with $c_1(M)>0.$ Here we
only discuss compact manifolds.

Since Ricci flow was introduced by R. Hamilton in \cite{[Ha82]},
Ricci solitons have been studied extensively. One interesting
question is to classify all the Ricci solitons with some curvature
conditions, especially with positive curvature operator. It is well
known that there is no expanding or steady Ricci solitons on a
compact Riemannian manifold of any dimension, and no shrinking Ricci
solitons of dimension 2 and 3 (cf. \cite{[Ha88]}\cite{[Ivey]}).
Recently B\"ohm-Wilking in \cite{[BW]} extended Hamilton's maximum
principles, and they essentially proved that there is no nontrivial
shrinking Ricci solitons with positive curvature operator on a
compact manifold.

For the K\"ahler case, by Siu-Yau's result in \cite{[SY]} a
K\"ahler-Ricci soliton with positive holomorphic bisectional
curvature is K\"ahler-Einstein. It is still very interesting to know
how to prove this result via K\"ahler-Ricci flow (cf.
\cite{[Wang]}), or via complex Monge-Amp\'ere equations.
Rotationally symmetric K\"ahler-Ricci solitons have been constructed
by Koiso \cite{[Koiso]} and Cao \cite{[Cao]}. The existence and
uniqueness of K\"ahler-Ricci solitons have been extensively studied
in literature (cf. \cite{[CTZ]}\cite{[TZ1]}\cite{[TZ2]}\cite{[WZ]}).

Let $M$ be a compact K\"ahler manifold with $c_1(M)>0.$ A K\"ahler
metric $g$ with its K\"ahler form $\oo$ is called a K\"ahler-Ricci
solition with respect to a holomorphic vector field $X$ if the
equation
$$Ric(\oo)-\oo=L_X \oo$$
is satisfied. Since $\oo$ is closed, we may write
$$L_X\oo=-\pbp u$$
for some function $u$ with $u_{ij}=u_{\bar i\bar j}=0.$ Then the
Futaki invariant is given by
$$f_X=-\frac 1V\int_M\;Xu\oo^n=\frac 1V\int_M\; |\Na u|^2\oo^n>0. $$
Note that $f_X$ depends only on the K\"ahler class $[\oo]$ and the
holomorphic vector field $X.$

The following is the first result in this paper:

\begin{thm}\label{main1}Let $\oo$ be a K\"ahler-Ricci soliton with a holomorphic vector field
$X$. If \begin{equation} |Ric(\oo)-\oo|<\frac
{-f_X+\sqrt{f_X^2+4f_X}}{2},\label{intro1}\end{equation} then $\oo$
is K\"ahler-Einstein.
\end{thm}

\begin{rem}Theorem \ref{main1} tells us that the Ricci curvature of
a K\"ahler-Ricci soliton can not be too close to the K\"ahler form,
so it gives a gap between K\"ahler-Ricci solitons and
K\"ahler-Einstein metrics.
\end{rem}

In \cite{[Tian97]}, G. Tian proposed a conjecture on the solution
$\oo_t$ of complex Monge-Amp\'ere equations on a K\"ahler manifold
with $c_1(M)>0$ \begin{equation} (\oo+\pbp
\varphi)^n=e^{h_{\oo}-t\varphi}\oo^n, \qquad \oo+\pbp
\varphi>0,\label{0.1}\end{equation} where $h_{\oo}$ is the Ricci
potential with respect to the metric $\oo.$ It is known that $M$
admits a K\"ahler-Einstein metric if and only if (\ref{0.1}) is
solvable for $t\in [0, 1]$. On the other hand, if $M$ admits no
K\"ahler-Einstein metrics, then (\ref{0.1}) is solvable for $t\in
[0, t_0)(t_0\leq 1)$. Tian conjectured that when this case occurs,
$(M, \oo_t)$ converges to a space $(M_{\infty}, \oo_{\infty})$,
which might be a K\"ahler-Ricci soliton after certain normalization.
Observe that the metric $\oo_t$ of (\ref{0.1}) has
$Ric(\oo_t)>(1-t)\oo_t$, the following theorem tells us that if
Tian's conjecture is true, $t_0$ might not be close to $1$.

\begin{thm}\label{main2}Let $\oo$ be a K\"ahler-Ricci soliton with a holomorphic vector field
$X$. There exists a constant $\ee>0$ depending only on the Futaki
invariant $f_X$ such that if
$$Ric(\oo)>(1-\ee)\oo,$$ then $\oo$ is K\"ahler-Einstein.
\end{thm}
\begin{rem}
In Theorem \ref{main2}, $\ee$ can be expressed explicitly from the
proof.
\end{rem}

{\bf Acknowledgements}: I would like to thank
 Professors F. Pacard for the support during the course of the work, and Professor X. X. Chen, W. Y. Ding and X. H. Zhu
 for their help and enlightening discussions.

\section{Proof of Theorem \ref{main1}}
In this section, we prove Theorem \ref{main1}. Let $g$ be a
K\"ahler-Ricci soliton with
$$Ric-\oo=-\pbp u.$$
By the definition of the Futaki invariant,
$$f_X=-\frac 1V\int_M\; X(u) \oo^n=\frac 1V\int_M\; |\Na u|^2\oo^n>0.$$
Define \begin{equation}
\ee:=\max_M\;|Ric-\oo|,\label{eq3}\end{equation} we have the
following lemma:

\begin{lem}\label{alem1}Let $\la_1$ be the first eigenvalue of $\Delta_g$, then $\la_1\geq 1-\ee$.
\end{lem}
\begin{proof}Let $f$ be an eigenfunction satisfying $\Delta_g f=-\la_1
f$, then \beqs 0\leq \int_M\; |\Na\Na f|^2\oo^n=\int_M\;
\Big((\Delta_g f)^2-Ric(\Na f, \bar \Na f) \Big)\oo^n\leq
(\la_1^2-(1-\ee)\la_1)\int_M\; f^2 \oo^n.\eeqs Thus, $\la_1\geq
1-\ee$ and the lemma is proved.
\end{proof}

By Lemma \ref{alem1}, we have \begin{equation} \frac 1V\int_M\;
|\Na\bar\Na u|^2\oo^n\geq (1-\ee)\frac 1V\int_M\; |\Na
u|^2\oo^n=(1-\ee)f_X.\label{sec1.1}\end{equation} In fact, we can
write $u=\sum_{i=1}^{\infty}\; c_i f_i$, where $f_i$ are the
eigenfunctions of $\Delta_g$ such that $$\Delta_g
f_i=-\la_if_i,\qquad \int_M\;f_i^2\oo^n=V.$$ Here
$0<\la_1\leq\la_2\leq \cdots \leq \la_m<\cdots.$ Then
$$\Delta_g u=-\sum_{i=1}^{\infty}\;c_i\la_if_i.$$ Integrating by parts, we have
$$\frac 1V\int_M\; |\Na\bar \Na u|^2\oo^n=\frac 1V\int_M\; (\Delta_g u)^2\oo^n=\sum_{i=1}^{\infty}\;c_i^2\la_i^2
\geq \la_1\sum_{i=1}^{\infty}\; c_i^2\la_i=\frac 1V\int_M\; |\Na
u|^2\oo^n.$$ Hence the inequality (\ref{sec1.1}) holds. On the other
hand, by (\ref{eq3}) we have
$$\frac 1V\int_M\; |\Na\bar\Na u|^2\oo^n\leq \ee^2.$$
Combining this together with (\ref{sec1.1}), we have
$$\ee^2-(1-\ee)f_X\geq 0.$$
Then
$$\ee\geq \frac {-f_X+\sqrt{f_X^2+4f_X}}{2}.$$
The theorem is proved.

\section{Proof of Theorem \ref{main2}}
In this section, we prove Theorem \ref{main2}. Let $g$ be a K\"ahler
Ricci soliton with
$$Ric(\oo)-\oo=-\pbp u,$$
where $u_{ij}=u_{\bar i\bar j}=0.$ We can normalize $u$ such that
\begin{equation} \int_M\; u\;\oo^n=0.\label{x2} \end{equation}

\begin{lem}\label{lem1} $u$ satisfies the following equation
\begin{equation} \Delta_g u+u-|\Na u|^2=-f_X. \label{x1}\end{equation}
\end{lem}
\begin{proof}By direct calculation, we have
 \beqs (u_{i\bar i}+u-u_iu_{\bar
i})_j=u_{i\bar ij}+u_j-u_iu_{\bar ij} =-R_{j\bar
k}u_k+u_j-u_iu_{\bar ij} =0. \eeqs Then $\Delta_g u+u-|\Na u|^2$ is
a constant. By (\ref{x2}) and the definition of the Futaki
invariant, (\ref{x1}) holds and the lemma is proved.
\end{proof}

Now we assume $Ric(\oo)\geq \la \, \oo$ with $\la>0,$ so the scalar
curvature $R\geq n\la.$ By Lemma \ref{lem1}, $u$ is bounded from
below. In fact, \begin{equation} u=-f_X+|\Na u|^2-\Delta u=-f_X+|\Na
u|^2+R-n\geq -f_X+n\la-n.\label{u lower}\end{equation} Now we can
prove the following lemma

\begin{lem}\label{lem2}The scalar curvature is uniformly bounded from above, i.e.
$$R\leq \La(\la, f_X),$$ where $\La(\la, f_X)$ is a constant depending
only on $\la$ and $f_X.$ Moreover, $\lim_{\la\ri 1}\La(\la, f_X)$ is
finite.
\end{lem}
\begin{proof}By the K\"ahler-Ricci soliton equation (\ref{x1}), we
have \begin{equation} |\Na u|^2=\Delta_g u+u+f_X=n-R+u+f_X\leq
u+n+f_X-n\la.\label{u grad}\end{equation} Let $B=f_X-n\la+n+1,$ then
by (\ref{u lower}) $u+B\geq 1.$ Hence
$$\frac {|\Na u|^2}{u+B}\leq 1-\frac {1}{u+B}\leq 1. $$
Let $p\in M$ be a minimum point of $u$, by the normalization
condition (\ref{x2}) $u(p)\leq 0.$ Then for any $x\in M,$ we have
\begin{equation} \sqrt{u+B}(x)-\sqrt{u+B}(p)\leq|\Na \sqrt{u+B}|\;diam(g)\leq
\frac 12 \frac {|\Na u|}{\sqrt{u+B}}\frac {c_n\pi}{\sqrt{\la}}\leq
\frac {c_n\pi}{2\sqrt{\la}},\label{u upper}\end{equation} where
$diam(g)$ is the diameter of $(M, g)$ and $c_n$ is a constant
depending only on $n.$ It follows that
$$u\leq \frac {c_n^2\pi^2}{4\la}+c_n\pi \sqrt{\frac {B}{\la}}.$$ Therefore,
$$ R=n-\Delta u=n+f_X-|\Na u|^2+u \leq n+f_X+\frac {c_n^2\pi^2}{4\la}+c_n\pi \sqrt{\frac {B}{\la}}.
$$

\end{proof}
Now we can finish the proof of Theorem \ref{main2}. By Lemma
\ref{lem2} and $R\geq n\la$, we have  \beqs \frac 1V\int_M\; |\Na
\bar \Na u|^2\oo^n&=&\frac 1V\int_M\;
(R-n)^2\oo^n\\
&=&\frac 1V\int_{\{R>n\}}(R-n)^2\oo^n+\frac 1V\int_{\{R<n\}}(R-n)^2\oo^n\\
&\leq &(\La-n)\,\frac 1V\int_{\{R>n\}}(R-n)\oo^n+n^2(1-\la)^2.\eeqs
On the other hand,
$$0=\int_{\{R>n\}}(R-n)\oo^n+\int_{\{R<n\}}\;
(R-n)\oo^n.$$ Therefore, \beqn \frac 1V\int_M\; |\Na \bar \Na
u|^2\oo^n&=&(\La-n)\frac 1V\int_{\{R<n\}}\;
(n-R)\oo^n+n^2(1-\la)^2\nonumber\\&\leq&
(\La-n)\,n(1-\la)+n^2(1-\la)^2\nonumber\\&\ri& 0, \label{x3}\eeqn as
$\la\ri 1.$ On the other hand, by the inequality (\ref{sec1.1}) we
have
$$\frac 1V\int_M\; |\Na \bar \Na u|^2\oo^n\geq \la f_X>0,$$
which contradicts (\ref{x3}) when $\la $ is sufficiently close to
$1.$ The theorem is proved.

\bigskip
\noindent Universit\'{e} Paris 12 - Val de Marne, 61, avenue du
G\'{e}n\'{e}ral de Gaulle, 94010 Cr\'{e}teil,  France\\
Email: lihaozhao@gmail.com


\begin{thebibliography}{1}
\bibitem{[BW]}C. B\"ohm, B. Wilking: Manifolds with positive curvature operators are space
forms, math.DG/0606187.

\bibitem{[Cao]} H. D. Cao: Existence of gradient Kahler-Ricci soliton, Elliptic and
parabolic methods in geometry, Eds. B.Chow, R.Gulliver, S.Levy,
J.Sullivan, A K Peters, pp.1-6, 1996.



\bibitem{[CTZ]}H. D.  Cao, G. Tian, X. H. Zhu:  K\"ahler-Ricci solitons on compact complex manifolds with
$C_1(M)>0,$ Geom and Funct. Anal., 15 (2005), 697-719.



\bibitem{[Ha82]}R. S. Hamilton: Three-manifolds with positive Ricci
curvature,  J. Diff. Geom.,  17(1982), no. 2, 255-306.
\bibitem{[Ha88]}R. S. Hamilton: The Ricci flow on surfaces.
Mathematics and general relativity, 237-262, Comtemp. Math., 71,
Amer. Math. Soc., Providence, RI, 1988.


\bibitem{[Ivey]}T. Ivey: Ricci solitons on compact three-manifolds, Diff. Geom. Appl., 3 (1993),
301-307.
\bibitem{[Koiso]}N. Koiso: On rotationally symmmetric Hamilton's equation for K\"ahler-Einstein
metrics, Recent Topics in Diff. Anal. Geom, Adv. Studies Pure
Math,18-I Academic Press Boston MA, 1990, 327-337.
\bibitem{[Pere]}G. Perelman: Unpublished work on K\"ahler-Ricci flow.

\bibitem{[NT]}N. Sesum, G. Tian: Bounding scalar curvature and diameter
 along the K\"ahler-Ricci flow (after Perelman) and some
 applications.



\bibitem{[SY]}Y. T. Siu,
S. T. Yau: Compact K\"ahler manifolds of positive bisectional
curvature, Invent. Math. 59 (1980), 189-204.



\bibitem{[Tian97]}G. Tian: K\"ahler-Einstein metrics with positive
scalar curvature,  Invent. Math.  130  (1997),  no. 1, 1-37.

\bibitem{[TZ1]}G. Tian, X. H. Zhu: Uniqueness of K\"ahler-Ricci solitons, Acta Math., 184 (2000), 271-305.

\bibitem{[TZ2]}G. Tian, X. H. Zhu: A new holomorphic invariant
and uniqueness of K\"ahler-Ricci solitons. Comment. Math. Helv.
77(2002), 297-325.
\bibitem{[Wang]}Y. Q. Wang: A New Ricci Flow Proof of Frankel
Conjecture. math.DG/0608151v1.
\bibitem{[WZ]}X. J. Wang, X. H. Zhu: K\"ahler-Ricci solitons on toric manifolds with positive first Chern
class, Adv. Math., 188 (2004), no. 1, 87-103.


\end{thebibliography}
\end{document}